\def\draft{n}
\documentclass{amsart}
\usepackage[headings]{fullpage}
\usepackage{amssymb,epic,eepic,epsfig}
\usepackage{pb-diagram}

\theoremstyle{plain}

\newtheorem{theorem}{Theorem}
\newtheorem{proposition}{Proposition}[section]
\newtheorem{lemma}[proposition]{Lemma}

\newtheorem{conjecture}{Conjecture}

\theoremstyle{definition}
\newtheorem{definition}[proposition]{Definition}

\newtheorem{question}{Question}

\theoremstyle{remark}
\newtheorem{example}[proposition]{Example}

\newtheorem{remark}[proposition]{Remark}

\def\printname#1{
        \if\draft y
                \smash{\makebox[0pt]{\hspace{-0.5in}
                        \raisebox{8pt}{\tt\tiny #1}}}
        \fi
}

\newcommand{\psdraw}[2]
         {\begin{array}{c} \hspace{-1.3mm}
        \raisebox{-4pt}{\epsfig{figure=draws/#1.eps,width=#2}}
        \hspace{-1.9mm}\end{array}}

\newlength{\standardunitlength}
\catcode`\@=11
\long\def\@makecaption#1#2{%
     \vskip 10pt

\setbox\@tempboxa\hbox{
       \small\sf{\bfcaptionfont #1. }\ignorespaces #2}%
     \ifdim \wd\@tempboxa >\captionwidth {%
         \rightskip=\@captionmargin\leftskip=\@captionmargin
         \unhbox\@tempboxa\par}%
       \else
         \hbox to\hsize{\hfil\box\@tempboxa\hfil}%
     \fi}
\font\bfcaptionfont=cmssbx10 scaled \magstephalf
\newdimen\@captionmargin\@captionmargin=2\parindent
\newdimen\captionwidth\captionwidth=\hsize
\catcode`\@=12

\newcommand{\Span}{\operatorname{Span}}

\def\lbl#1{\label{#1}\printname{#1}}

\def\BN{\mathbb N}
\def\BZ{\mathbb Z}

\def\BQ{\mathbb Q}

\def\BC{\mathbb C}

\def\T{\mathcal T}

\def\T{\mathcal T}

\def\calV{\mathcal V}

\def\La{\Lambda}

\def\la{\langle}
\def\ra{\rangle}

\def\longto{\longrightarrow}
\def\pt{\partial}

\def\Res{\mathrm{Res}}
\def\Cone{\mathrm{Cone}}
\def\cone{\mathrm{cone}}
\def\rot{\mathrm{rot}}
\def\Vol{\mathrm{Vol}}

\def\calL{\mathcal{L}}

\def\calF{\mathcal{F}}
\def\calC{\mathcal{C}}
\def\calM{\mathcal{M}}

\def\lin{\mathrm{lin}}
\def\Fl{\mathrm{Fl}}
\def\Supp{\mathrm{Supp}}
\def\T{\mathrm{T}}
\def\vol{\mathrm{vol}}
\def\lin{\mathrm{lin}}
\def\Tan{\mathrm{Tan}}
\def\l{\lambda}
\def\La{\Lambda}
\def\lin{\mathrm{lin}}
\def\aff{\mathrm{aff}}
\def\SI{\rm{SI}}
\def\IS{\rm{IS}}
\def\ISo{$\mathrm{IS}^{0}$}

\begin{document}


\title[Sum-integral interpolators and the Euler-Maclaurin formula
for polytopes]{Sum-integral interpolators and the Euler-Maclaurin formula
for polytopes}
\author{Stavros Garoufalidis}
\address{School of Mathematics \\
         Georgia Institute of Technology \\
         Atlanta, GA 30332-0160, USA \\ 
         {\tt http://www.math.gatech} \newline {\tt .edu/$\sim$stavros } }
\email{stavros@math.gatech.edu}
\author{James Pommersheim}
\address{Mathematics Department \\
         Reed College \\
         3203 SE Woodstock Boulevard \\
Portland, Oregon 97202-8199}
\email{jamie@reed.edu}

\thanks{S.G. was supported in part by NSF. \\
\newline
1991 {\em Mathematics Classification.} Primary 57N10. Secondary 57M25.
\newline
{\em Key words and phrases: polytopes, Euler-MacLaurin formula, 
reverse Euler-MacLaurin formula, 
lattice points, flag varieties, exponential sums, exponential integrals,
interpolators. 
}
}

\date{May 20, 2010 }


\begin{abstract}
A local lattice point counting formula, and more generally a local
Euler-Maclaurin formula follow by  
comparing two natural families
of meromorphic functions on the dual of a rational vector space $V$,
namely the family of exponential sums (S) and the family of
exponential integrals (I) parametrized by the set of rational polytopes in 
$V$. 
The paper introduces the notion of an interpolator between these two families
of meromorphic functions.  We prove that every rigid complement map in $V$ 
gives rise to an effectively computable 
\SI-interpolator (and a local Euler-MacLaurin formula), 
an \IS-interpolator (and a reverse local Euler-MacLaurin formula) and an 
\ISo-interpolator (which interpolates between integrals and sums over interior lattice points.) 
Rigid complement maps can be constructed by choosing an inner product on
$V$ or by choosing a complete flag in $V$. 
The corresponding interpolators generalize and unify the work of 
Berline-Vergne,  Pommersheim-Thomas, and Morelli. 
\end{abstract}

\maketitle

\tableofcontents

\section{Introduction}
\lbl{sec.intro}

\subsection{What is a local lattice-point counting formula?}
\lbl{sub.previous}

The relationship between sums and integrals has been of interest to 
mathematicians since the ancient Greeks.  The classical Euler-Maclaurin 
formula, discovered in the first half of the eighteenth century shortly 
after the development of modern calculus, may be viewed as a relationship 
between 
the sum of a function over the lattice points in a one-dimensional polytope 
with integer vertices and the integral of the function over the polytope. 
One naturally asks the same question in higher dimensions: Given a polytope 
$P$ in an $n$-dimensional space $V$ equipped with an $n$-dimensional lattice 
$\Lambda$ and a function $f$ on $V$, can one express the sum of $f$ over 
the lattice points in $P$ in terms of the integral $f$ over $P$?  In such a 
formula, one would expect a main term involving the integral of $f$ over 
$P$ as well as correction terms involving the integrals of $f$ over the 
proper faces of $F\subset P$.

In the simplest case, let us suppose that $f$ is a constant function. Then, 
the question becomes that of expressing the number of lattice points in $P$ 
in terms of volume of $P$ and the volumes of the faces $F$ of $P$. If $P$
is 2-dimensional, the celebrated {\em Pick's formula} \cite{Pi}
$$
\#(P)=A+\frac{1}{2} b +1
$$
expresses the number $\#(P)$ of lattice points in a convex lattice polygon 
in terms of its area $A=\vol(P)$, and the number of lattice points 
$b=\vol(\pt P)$ of its boundary. For example, we have: 
$$
\psdraw{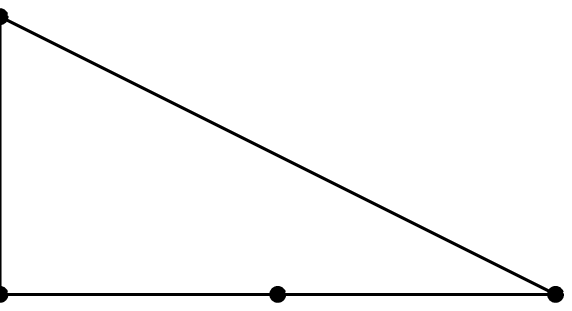}{1.0in} \qquad\qquad
\#(P)=4, A=1, b=4.
$$
Here and throughout, we follow the usual convention that 
all volumes $\vol(F)$ of faces are normalized so that a lattice basis of 
$\Lambda \cap L$, where
$L$ is the  linear space parallel to $F$, has volume $1$.  
Unfortunately, Pick's formula is not {\em local}, due to the presence of 
the term $1$.
Here, locality means that for each face $F$ of $P$,  the formula contains 
a term that is the volume $\vol(F)$ multiplied by a coefficient that 
depends only on the {\it supporting cone} $\Supp(P,F)$ to $P$ along $F$. 
The supporting cone is defined as the union of rays whose 
endpoint is in $F$ and which remain in $P$ for a positive distance.  See 
Section \ref{sub.polytopes} below for a precise definition. McMullen 
\cite{McM} proved the existence of local lattice point counting formulas.
More precisely, he proved the existence of 
a function $\mu$ from rational, convex 
cones to rational numbers such that for any integral polytope $P$, the 
number $\#(P)$ of lattice points in $P$ is given by 
\begin{equation}
\lbl{eq.mcmullen}
\# (P) = \sum_{F} \mu(\Supp(P,F))\vol(F).
\end{equation}
Here the sum is taken over all faces $F$ of $P$, and $\vol(P)$ denotes the 
volume of the face $F$.   
 
Part of the difficulty constructing and computing 
a function $\mu$ that satisfies \eqref{eq.mcmullen} is that $\mu$ is far 
from unique. The second author 
and Thomas gave an explicit construction of a rational valued function
$\mu$ satisfying \eqref{eq.mcmullen}, given a fixed {\em complement map},
a notion introduced in Thomas's thesis (cf. \cite{Th}).  A complement map
is a systematic choice of complements of linear subspaces of
a vector space; see Section \ref{sub.cmap} for a 
precise definition.  
Also note that all of the complement maps in this paper will be {\em rigid}.
Two natural ways to get a rigid complement map are to choose: (1) an inner 
product, 
or (2) a complete flag. Given an inner product or a complete flag, \cite{PT}
construct effectively a map $\mu$ that satisfies Equation \eqref{eq.mcmullen}.
For the triangle depicted above, the Pommersheim-Thomas values of $\mu$ at 
the supporting cones to the
vertices for an inner products or complete flags are as follows:
$$
\psdraw{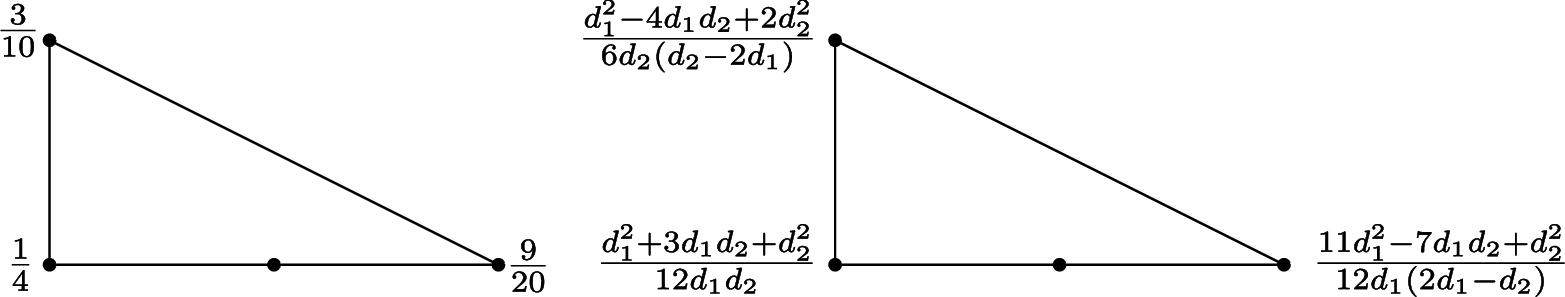}{5in}
$$
The values on the left arise from the standard inner product on $\BZ^2$, 
according to \cite[Cor.1]{PT}, and may be computed by
multiplying out the Todd polynomial in the ring presentation given in 
\cite[Prop.2]{PT}.  Readers wishing more details on this should 
consult Section \ref{sub.examples}.
The values in the figure on the right are those arising from the choice of 
the complete flag in $V^* = \BZ^2$
whose 1-dimensional subspace is spanned by the point $(d_1,d_2)$. These 
values can be computed using the
method outlined in \cite[p.198]{Mo}.  They can also be computed by 
multiplying out the Todd polynomials using \cite[Thm.3]{PT}.
One feature of this 
construction is that if one chooses a complement map arising from a complete 
flag, one recovers exactly the lattice point
formulas of Morelli \cite{Mo}, who gave a function $\mu$
that satisfies Equation 
\eqref{eq.mcmullen} and takes values on the field of 
rational functions on a Grassmannian.  For example, in the triangle on 
the right above, the values shown are rational functions on the Grassmannian 
of $1$-dimensional subspaces of $V^*$. The construction of \cite{PT} 
is based on the theory of toric varieties, and gives an 
answer to a question of Danilov about the existence of a local expression 
for the Todd class of a toric variety.

\subsection{What is a local Euler-MacLaurin summation formula?}
\lbl{sub.localEM}

Returning to the Euler-Maclaurin question, Berline and 
Vergne constructed in \cite{BV} an explicit local Euler-MacLaurin 
formula for the sum of a polynomial function $f$ 
over the lattice points $P \cap \La$ of an $n$-dimensional polytope
$P$ in a rational vector space $V$ with lattice $\La$.
The Berline-Vergne formula has the form
\begin{equation}
\lbl{eq.BVEM}
\sum_{x \in P \cap \La} f(x)=\sum_{F}\int_{F} 
D(P,F) \cdot f
\end{equation}
where the sum is over the set of faces $F$ of $P$, and 
$D(P,F)$ an infinite-order constant-coefficient 
differential operator $D(P,F)$ that depends only on the supporting cone 
$\Supp(P,F)$. Equation \eqref{eq.BVEM} is a generalization of 
\eqref{eq.mcmullen}. Indeed, if $D(P,F)$ satisfies \eqref{eq.BVEM}, and
we define $\mu(\Supp(P,F))$ to be the constant term of $D(P,F)$, then
the local lattice point formula \eqref{eq.mcmullen} holds.

As in McMullen's case, the infinite order differential operators $D(P,F)$
are not uniquely determined by \eqref{eq.BVEM}. 
The construction of Berline-Vergne requires an inner 
product on the vector space $V$, and their results apply to rational 
polytopes, rather than just integral polytopes. Their construction 
associates to each cone $K$ in $V$ a meromorphic function $\mu(K)$ on 
the dual space that is regular at the origin.  For cones of dimension at most $2$,
the value of this function 
at 0 (the constant term of the operator $D(P,F)$) coincides with the 
$\mu$ constructed in \cite{PT} for the special case of complement maps arising 
from an inner product.  In this way, for the example triangle above, the Berline-Vergne construction
recovers the the numbers $\frac{1}{4}, \frac{9}{20}, \frac{3}{10}$ arising out of the
inner product case of the
Pommersheim-Thomas construction.  This fact that this coincidence holds is new
to this paper; see Theorem \ref{thm.1}, properties (5) and (6).

\subsection{An informal presentation of the results of this paper}
\lbl{sub.results}

The local Euler-MacLaurin formula \eqref{eq.BVEM} is a consequence of a 
relationship between the integral and the sum of an exponential function 
over a polytope.  In this paper, we introduce the concept of an 
{\it interpolator}  between the families of
exponential sums ($S$) and exponential integrals ($I$)
over rational polytopes in a rational
vector space $V$. We now discuss our results about interpolators 
informally, leaving the precise definitions and statements for Sections 
\ref{sec.polytopes} and \ref{sec.results}.

If $P$ is a rational polytope, or more generally a rational polyhedron 
(region, not necessarily compact, defined by linear inequalities with 
rational coefficients) in a rational vector space $V$, one can associate 
to $P$ two important meromorphic functions, the {\em exponential sum}
 $S(P) \in \calM(V^*)$ and the {\em exponential integral} $I(P) \in \calM(V^*)$
where $\calM(V^*)$ is the algebra of 
meromorphic functions on the dual space $V^*\otimes \BC$.  The values 
of these functions at a point $\xi\in V^* \otimes \BC$ are given by:
\begin{equation}
\lbl{eq.SandI}
S(P)(\xi)=\sum_{x \in P \cap \Lambda}  e^{\la \xi,x \ra}, 
\qquad
I(P)(\xi)=\int_{P}  e^{\la \xi,x \ra} dx
\end{equation}
provided $|e^{\la \xi,x \ra}|$ is summable (resp. 
integrable) over $P$.  Here, the integral is taken with respect to the 
relative Lebesgue measure on $\aff(P)$, normalized so that a basis 
of $\Lambda\cap W$, where $W$ is the linear subspace parallel to the 
affine span of $P$, has volume 1.
The fact that Equations \eqref{eq.SandI} define meromorphic functions, 
as well as the precise characterization and properties of the functions 
$S$ and $I$, is essentially  the content of Lawrence's theorem \cite{La}, 
reviewed in Section \ref{sub.si} below. 

The local Euler-Macluarin formula \eqref{eq.BVEM} follows in a straighforward
manner from a formula of the following shape for a rational polyhedron $P$
\begin{equation}
\lbl{eq.interp}
S(P) = \sum_{F} \mu(\Supp(P,F)) I(F)
\end{equation}
where $\mu$ is a function on the set of cones in $V$ with values in 
$\calM(V^*)$.
We call such a function an {\it \SI-interpolator}, or simply an 
{\it interpolator}.  The main result of this paper (Theorem \ref{thm.1}) 
states that a complement map on the vector space $V$ gives rise 
in a natural way to an effectively computable 
\SI-interpolator on $V$, and hence a local Euler-Maclaurin formula of
the form \ref{eq.BVEM}; see Theorem \ref{thm.2}.

Thus, in particular, an inner product 
on $V$ or a complete flag in $V$ gives rise to 
a local Euler-Maclaurin formula.  These 
interpolators have interesting connections with previous 
results.  For complement maps arising from the choice of an inner product, 
one recovers the $\mu$ constructed in \cite{BV}. In addition, we show that 
the values of these functions $\mu$
at 0 are given 
by the functions $\mu$ in \cite{PT} for cones of dimension at most $2$, and conjecture that
these coincidences hold in all dimensions.  In the case of complete flags, one 
obtains interpolators $\mu$ that are entirely new, as can be seen by 
calculations in dimension $2$.  In this case,  
the values of the constant 
term $\mu(K)(0)$ coincide with values constructed by Morelli for cones of dimension at most $2$,
and conjecturally in all dimensions.
  One can 
also vary the chosen flag, thus associating to each cone $K$ a meromorphic 
function $\mu(K)$ on $\Fl(V^*)\times V^*$, where $\Fl(V^*)$
is the complete flag variety of $V$.  This function $\mu(K)$
is naturally defined and effectively computable {\em independent} of any 
choices; see Theorem \ref{thm.3}.  
In Morelli's work, he constructs what amounts to the constant term 
$\mu(K)(0)$ of this meromorphic function. Working dimension-by-dimension, 
Morelli chooses to view this term as a rational function on a Grassmannian, 
rather than on the entire flag variety.
 
In Section \ref{sub.IS}, we show that a complement map on 
$V$ also leads naturally to the construction of a local \IS-interpolator, 
expressing the exponential integral $I(P)$ in terms of the exponential 
sums $S(F)$, for $F$ a face of $P$; see Theorem \ref{thm.4}.   This 
construction, which works only in the case of integral polyhedra, allows 
one to obtain a
{\em reverse local Euler-MacLaurin formula} expressing the integral of a 
polynomial function over a polytope in terms of the sums of the function 
over lattice points in the various faces of the polytope; see Theorem 
\ref{thm.22}.   
 
Finally, we show how a complement map yields an 
\ISo-interpolator, where we use $\mathrm{S}^0$ to denote the sum over 
interior lattice points of a polytope.  Part of the interest in 
\ISo-interpolators lies in an observation of Morelli, who 
asked essentially (in the language of the present paper) if the same 
function can simultaneously serve as the constant term of an 
\SI-interpolator and as the constant term of an \ISo-interpolator 
on the dual space.  He observed that the constant terms that he constructs 
do exactly this in dimensions at most 4.  We give a proof of these 
coincidences in dimensions $1$ and $2$, and extend these results to the 
case of complement maps arising from inner products.  Our arguments are 
essentially geometric, in contrast to Morelli's rather algebraic argument. 
(See Sections \ref{sub.morelli} and \ref{sec.morelli} for a complete 
discussion.)

Section \ref{sec.examples} contains explicit computations of the 
functions $\mu(K)$ for some cones $K$ of dimensions at most $2$.  We 
compute these functions for complement maps arising both from inner 
products and from complete flags.  In Section \ref{sub.examples}, we 
exhibit Equation \ref{eq.interp} in the case of two triangles including the 
triangle depicted
above. for both the inner product and complete 
flag cases, matching the constant terms with those constructed by 
Pommersheim-Thomas and Morelli shown in the above figure.

\section{Polytopes, exponential sums/integrals and interpolators}
\lbl{sec.polytopes}

\subsection{Polytopes}
\lbl{sub.polytopes}

We will use standard terminology for rational polytopes, polyhedra, cones
and faces, following for example \cite{Fu}. We will fix a finite dimensional
lattice $\La$, isomorphic to $\BZ^k$ for some $k \in \BN$, and consider 
the rational vector space $V=\La\otimes_{\BZ}\BQ$.
All polyhedra will be rational and convex, and all cones will be rational, 
polyhedral and may be affine. A {\em polytope} is a compact polyhedron.
If $P$ is a polyhedron and $F$ is a face of $P$, then the 
{\it tangent cone} $\Tan(P,F)$ is defined by picking a point $x$ in the 
relative 
interior of $F$ and looking at all the directions that one can go and stay 
in $P$:
$$
\Tan(P,F) = \{ v\in V | x +\epsilon v \in P \rm{\ for\ small\ } \epsilon > 0 \}
$$
This cone is independent of the choice of $x$ and contains the origin. 
Let $\aff(F)$ denote the {\em affine
span} of $F$, i.e., the smallest affine subspace of $V$ that contains $F$.
Let $\lin(F)$ denote the linear subspace of $V$ parallel to $\aff(F)$.
$\lin(F)$ is the maximal linear subspace of $V$ contained in $\Tan(P,F)$.
The {\it supporting cone} $\Supp(P,F)$ 
is the tangent cone translated back to its original position:
$$
\Supp(P,F) = \Tan(P,F) + x
$$
for $x$ in the relative interior of $F$.   
$$
\psdraw{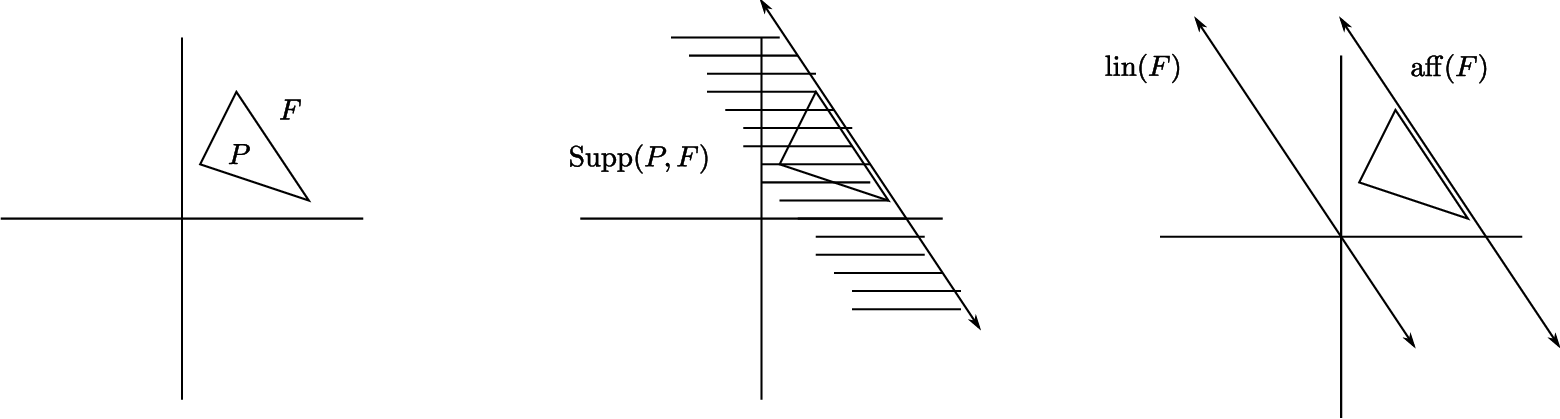}{5in}
$$

\subsection{Exponential sums and integrals}
\lbl{sub.si}

Let $V^*=\mathrm{Hom}_{\BQ}(V,\BQ)$ denote the dual vector space to $V$, and
$\la \cdot, \cdot \ra : V^* \times V \longto \BQ$ denote the natural
evalutation pairing on $V$. By a {\em meromorphic function} on $V^*$
we will mean a meromorphic function on the  complexified dual space 
$V_{\BC}^*$, where $V_{\BC}=V \otimes_{\BQ} \BC$ denotes 
the complexification of $V$.  Let $\calM(V^*)$ denote the algebra of 
meromorphic functions on $V^*$ and $\calM^r(V^*)$ denote the subalgebra 
of meromorphic functions regular (i.e., analytic) at the origin. 
Note that if $W$ is a rational subspace of $V$, then $\La'=\La\cap W$
is a full rank lattice in $W$ which naturally induces a relative Lebesgue 
measure on $W$, normalized so that a basis of $\La'$ has volume $1$. Likewise,
every rational affine subspace of $V$ has a natural measure induced by
the lattice.

Given a polyhedron $P$ in a rational vector space $V$
with lattice $\Lambda$ as above,  
there are two important meromorphic functions $S(P)$ and $I(P)$ 
with rational coefficients in $V^*$, namely the
generating function of {\em exponential sums} and {\em exponential
integrals}. 

These functions are uniquely characterized by the following 
properties. 

\begin{itemize}
\item[(A1)] 
If $P$ contains a straight line, then $S(P)=I(P)=0$.
\item[(A2)]
$S$ (resp. $I$) is a {\em valuation} (resp. a {\em simple (or solid) valuation}). That is, 
if the characteristic functions of a family of polyhedra satisfy a relation
$\sum_i r_i \chi(P_i)=0$, then the functions $S(P_i)$ satisfy
the relation 
$\sum_i r_i S(P_i)=0$ (resp. restrict the sum to those $P_i$ that do 
not lie in a proper affine subspace of $V$.)
\item[(A3)]
For every $s \in V$ and every $\xi \in V^*$ we have 
\begin{equation}
\lbl{eq.A3a}
I(s+P)(\xi)=e^{\la \xi,s \ra} I(P)(\xi), 
\end{equation}
and 
\begin{equation}
S(s+P)(\xi)=e^{\la \xi,s \ra} S(P)(\xi), \qquad s \in \La.
\end{equation}
\item[(A4)]
If $\xi \in V^*$ is such that $|e^{\la \xi,x \ra}|$ is integrable (resp. 
absolutely summable) over $P$, then
\begin{equation}
\lbl{eq.A4}
I(P)(\xi)=\int_{P}  e^{\la \xi,x \ra} dm_P(x), 
\qquad
S(P)(\xi)=\sum_{x \in P \cap \Lambda}  e^{\la \xi,x \ra}
\end{equation}
where $dm_P$ denotes the relative Lebesgue measure on $\aff(P)$.
\end{itemize}
The existence of the functions $I$ and $S$ is by no means obvious, and is
essentially the content of Lawrence's theorem; \cite{La}. (A1)-(A4)
determine the value of $I$ and $S$ on a simplicial cone as follows.

\begin{itemize}
\item[(A5)]
If $K=\Cone(v_1,\dots,v_k)$ is a simplicial cone generated by $k$ 
independent vectors $v_1,\dots,v_k$, (where $k=\dim(V)$ in the case of $S$) 
and $\Box(v_1,\dots,v_k)$ denote the parallelepiped generated by 
$v_1,\dots,v_k$, we have:
\begin{equation}
\lbl{eq.A5}
I(K)(\xi)=(-1)^k \frac{\text{vol}(\Box(v_1,\dots,v_k))}{\prod_{i=1}^k
\la \xi, v_i \ra},
\qquad
S(K)(\xi)=\left(\sum_{x \in \Box(v_1,\dots,v_k) \cap \Lambda} 
e^{\la \xi,x \ra} \right)
\prod_{i=1}^k  \frac{1}{1-e^{\la \xi,v_i \ra}},
\end{equation}
\end{itemize}
Since any polyhedron can be subdivided (or virtually subdivided, using a linear combination of characteristic functions) 
into a union of simplicial cones, together with Equation \eqref{eq.A5} above,
one obtains an algorithm to compute $I(P)$ and $S(P)$. The 
complexity of such an algorithm is discussed in \cite{Bv}.
The next theorem which computes $I(P)$ and $S(P)$ of a polyhedron in terms
of the tangent cones of its vertices was obtained by Brion 
(using toric varieties) and Lawrence (using combinatorics) independently.

\begin{proposition}
\lbl{prop.BL}\cite{Br,La}
Let $P$ be a polyhedron in $V$. Then,
\begin{equation}
I(P)=\sum_{v \in \calV(P)}I(\Supp(P,v)),
\qquad
S(P)=\sum_{v \in \calV(P)}S(\Supp(P,v))
\end{equation}
where the sums are over the set $\calV(P)$ of vertices of $P$.
\end{proposition}

\subsection{\SI-Interpolators}
\lbl{sub.interpolate}

So far, our discussion of the meromorphic functions $S(P)$ and $I(P)$
was parallel but independent: our formulas did not mix $S$ and $I$. 
It is natural to ask for an interpolation between
$S(P)$ and $I(P)$. Indeed, if $P$ is a polytope, then letting $\xi=0$ reduces 
to a classic problem: expressing the number of lattice points in $P$ in 
terms of the volumes of the
faces of $P$. Let $\calC(V)$ denote the set of cones of $V$.

\begin{definition}
\lbl{def.interpolator}
An {\em SI-interpolator} (or simply, interpolator) on $V$ is map 
$$
\mu:\calC(V) \longto \calM(V^*)
$$  
such that for any rational polyhedron $P$ in $V$ we have:
\begin{equation}
\lbl{eq.inter}
S(P)=\sum_{F \in \calF(P)} \mu(\Supp(P,F))I(F)
\end{equation}
where the sum is over the set $\calF(P)$ of all faces of $P$.
An interpolator is {\em regular} if  $\mu$ takes values in $\calM^r(V^*)$.
\end{definition}

For the next definition, recall that if $W$ is a quotient of $V$
with projection $\pi: V \longto W$, then there is a natural map 
$\calC(W) \longto \calC(V)$ given by $K \mapsto \pi^{-1}(K)$. Moreover, 
$W^*$ is naturally a subspace of $V^*$, and there is a restriction map $
\calM(V^*) \longto \calM(W^*)$.

\begin{definition}
\lbl{def.interhereditary}
An \SI-interpolator $\mu$ on $V$ is {\em hereditary} if for every 
rational quotient $W$ of $V$, the composition 
$\bar{\mu}:\calC(W) \longto \calM(W^*)$ given by the following
diagram
$$
\begin{diagram}
\node{\calC(W)}\arrow{e}\arrow{s,l}{\bar\mu}\node{\calC(V)}\arrow{s,r}{\mu} \\
\node{\calM(W^*)}\node{\calM(V^*)}\arrow{w}
\end{diagram}
$$
is an interpolator on $W$.
\end{definition}

The next lemma reduces the checking of the interpolator equation 
\eqref{eq.inter} to the case of cones.

\begin{lemma}
\lbl{lem.interpolator}
$\mu$ is an interpolator if and only if it satisfies \eqref{eq.inter}
for all cones $P$ in $V$.
\end{lemma}

\begin{proof}
This follows from Proposition \ref{prop.BL}. Indeed, we have:
\begin{eqnarray*}
S(P) &=& \sum_{v \in \calV(P)} S(\Supp(P,v)) \\
&=& \sum_{v \in \calV(P)} \sum_{F' \in \calF(\Supp(P,v))}
\mu(\Supp(\Supp(P,v),F')) I(F').
\end{eqnarray*}
Since every face of $\Supp(\Supp(P,v),F')$ is a face of $\Supp(P,F)$
for some face $F$ of $P$, it follows that
$$
S(P)=\sum_{F \in \calF(P)} \mu(\Supp(P,F)) \sum_{v \in \calV(F)} I(
\Supp(F,v))=\sum_{F \in \calF(P)} \mu(\Supp(P,F)) I(F).
$$
The result follows.
\end{proof}

\subsection{Rigid Complement maps}
\lbl{sub.cmap}

A rigid complement map gives us a systematic way to extend a function
on a linear subspace of a vector space to the entire vector space.
This notion was introduced in the thesis of Thomas \cite{Th}
and used in \cite{PT}, where it is shown that a choice of 
complement map leads naturally to a local formula for the number of lattice 
points in a polytope and a local formula for the Todd class of a toric 
variety. In our paper, we will only use the notion of a {\em rigid
complement map}. The curious reader may consult \cite{PT,Th} for the definition
of the general complement maps.

\begin{definition}
\lbl{def.cmap}
A {\em rigid complement map} on $V^*$ is a map $\Psi$ from a collection
$\calL^{\Psi}$ of linear subspaces in $V^*$ to
the set of linear subspaces of $V^*$ satisfying two properties:
\begin{itemize}
\item[(a)] For every $U \in \calL^{\Psi}$, $\Psi(U)$ is complementary
to $U$; i.e., $\Psi(U) \cap U=\{0\}$ and $\Psi(U)+U=V^*$.
\item[(b)] If $U_1 \subset U_2 \subset V^*$ with $U_1,U_2 \in \calL^{\Psi}$, then
$\Psi(U_2) \subset \Psi(U_1)$.
\end{itemize}
\end{definition}
There are two easy ways to construct rigid complement maps: 
(a) an inner product $Q$ on $V^*$, and (b) a complete flag $L$ on $V^*$.

\begin{lemma}
\lbl{lem.examplesPsi}
\rm{(a)} An inner product $Q$ on $V^*$ defines a complement map $\Psi^Q$
with domain $\calL^Q$ the set of all subspaces of $V^*$ such that
$$
\Psi(U)=U^{\bot}
$$
where $U^{\bot}$ denotes the subspace of $ V^*$ perpendicular to $U \subset V^*$ under
the inner product $Q$.
\newline
\rm{(b)} A complete flag $L=(L_0,\dots,L_n)$ on $V^*$ satisfying 
$L_0=\{0\}\subset L_1 \subset \dots \subset L_n =V^*$ defines a complement map 
on $V^*$ by
$$
\Psi^L(U) = L_{n-\dim(U)}
$$
for any linear subspace $U$ of $V^*$ which is complementary to 
$L_{n-\dim(U)}$. This complement map is only defined generically, i.e., 
for subspaces that meet the flag generically.
\end{lemma}

Our main theorem, Theorem \ref{thm.1} below, requires cones that are generic 
with
respect to a rigid complement map $\Psi$. Let us define what generic means.
Fix a rigid complement map $\Psi$ on $V^*$.

\begin{definition}
\lbl{def.genericPsi}
\rm{(a)} A quotient $W$ of $V$ is $\Psi$-generic if $W^* \in \calL^{\Psi}$,
where $W^* \subset V^*$.
\newline
\rm{(b)} A subspace $U$ of $V$ is  $\Psi$-generic if $V/U$ is $\Psi$-generic.
\newline
\rm{(c)} A cone $K$ of $V$ is {\em $\Psi$-generic} if
$\lin(F)$ is $\Psi$-generic for every face $F$ of $K$.
\newline
\rm{(d)} A polyhedron $P$ of $V$ is {\em $\Psi$-generic} 
if for every face $F$ of $P$, $\Supp(P,F)$ is $\Psi$-generic.
\end{definition}

The next lemma states that the notion of a complement map on $V^*$ is
hereditary, i.e., a complement map on $V^*$ gives rise to a complement
map on every $W^*$, where $W$ is a $\Psi$-generic rational quotient of $V$. 

\begin{lemma}
\lbl{lem.cmapher}
If $\Psi$ is a rigid complement map on $V^*$ and $W$ is a $\Psi$-generic 
rational quotient of $V$, then there is a natural complement map $\bar\Psi$
on $W^*$ defined by 
$$
\bar\Psi(U)=\Psi(U) \cap W^*
$$
for all $U \subset W^*$ such that $\Psi(U) \cap W^*$ is complementary to $U$.
\end{lemma}

The next lemma gives the promised extension of functions on a linear subspace
$W^*$ to ones on the entire space $V^*$.

\begin{lemma}
\lbl{lem.cmapextend}
If $\Psi$ is a rigid complement map on $V^*$ and $W$ is a 
$\Psi$-generic quotient of $V$, then the decomposition
$$
V^* = W^* \oplus \Psi(W^*)
$$ 
defines a unique linear projection map
$$
\pi:V^*\longrightarrow W^*
$$ 
that annihilates $\Psi(W^*)$ and is the identity on $W^*$.   This allows any 
function on $W^*$ to be extended to $V^*$.  
\end{lemma}

We now come to a main definition. Fix a rigid complement map $\Psi$ on $V^*$.
Let $\calC^{\Psi}(V)$ denote the set of $\Psi$-generic cones of $V$.

\begin{definition}
\lbl{def.Psicompatible}   
A $\Psi$-{\em compatible} interpolator $\mu$ on $V$ is a function
$\mu: \calC^{\Psi}(V) \longto \calM(V^*)$ satisfying
\begin{itemize}
\item[(a)]
Equation \eqref{eq.inter} for all $\Psi$-generic polytopes $P$,
\item[(b)]
$\mu$ is $\Psi$-hereditary, i.e., 
for every $\Psi$-generic quotient $W$ of $V$, the composition 
$\bar{\mu}:\calC(W) \longto \calM(W^*)$ given by the following
diagram
\begin{equation}
\lbl{eq.comp1}
\begin{diagram}
\node{\calC^{\bar\Psi}(W)}\arrow{e}\arrow{s,l}{\bar\mu}
\node{\calC^{\Psi}(V)}\arrow{s,r}{\mu} \\
\node{\calM(W^*)}\node{\calM(V^*)}\arrow{w}
\end{diagram}
\end{equation}
satisfies Equation \eqref{eq.inter} for all $\bar\Psi$-generic polytopes
$P$ in $W$.
\item[(c)]
Moreover, the following diagram commutes:
\begin{equation}
\lbl{eq.comp2}
\begin{diagram}
\node{\calC^{\bar\Psi}(W)}\arrow{e}\arrow{s,l}{\bar\mu}
\node{\calC^{\Psi}(V)}\arrow{s,r}{\mu} \\
\node{\calM(W^*)}\arrow{e,t}{\pi}\node{\calM(V^*)}
\end{diagram}
\end{equation}
where $\pi$ is given in Lemma \ref{lem.cmapextend}.
\end{itemize}
\end{definition}

\section{Statement of the results}
\lbl{sec.results}

\subsection{\SI-interpolators and local Euler-MacLaurin formula}
\lbl{sub.SI}

Our main result is that a rigid complement map $\Psi$ on $V^*$
determines algorithmically a unique $\Psi$-compatible interpolator 
$\mu^\Psi$ on $V$. This generalizes results of \cite{BV,Mo,PT}. 
The proof uses the techniques of \cite{BV},
which were a motivation and inspiration for us.

\begin{theorem}
\lbl{thm.1}
\rm{(a)}
If $V$ is a rational vector space and  $\Psi$ is a rigid complement map on  
$V^*$, there is a unique $\Psi$-compatible interpolator  $\mu^{\Psi}$.
\newline
\rm{(b)}
In addition, $\mu^{\Psi}$ satisfies the following properties on the collection
of $\Psi$-generic cones on $V$:
\begin{itemize}
\item[(1)] (Additivity) 
If the characteristic functions of a finite collection of 
cones $K_i$ with vertex $v\in V$ satisfy the relation
$\sum_i r_i \chi(K_i)=0$, then  the functions 
$\mu^{\Psi}(K_i)$ satisfy
the relation 
$\sum_i r_i \mu^{\Psi}(K_i)=0$ .
\item[(2)] (Lattice Invariance)  If $v\in\Lambda$, then 
$\mu^\Psi(v+K) = \mu^\Psi(K)$. 
\item[(3)] (Isometry Equivariance)  If $g$ is a lattice-preserving linear
automorphism of $V$ and $g^*$ is the inverse transpose, then 
$\mu^{g \Psi}(g(K))(g^*(\xi)) = \mu^\Psi(K)(\xi)$. 
\item[(4)] (Regularity)
$\mu^{\Psi}$ is regular at $\xi=0$.
\end{itemize}
Moreover, $\mu^{\Psi}$ generalizes previous results of \cite{Mo} and 
\cite{BV}:
\begin{itemize}
\item[(5)] (Constant Term) Let $K$ be a top dimensional
cone in $V$ with vertex at 
zero, and suppose that $V$ has dimension at most 2. 
Then for $\Psi$ arising from an inner product or a complete flag,
the constant term of $\mu^\Psi(K)$ agrees with the Todd class 
coefficient $\mu(K^{\vee})$ of the dual cone $K^{\vee}$
that depends on $\Psi$, 
constructed in \cite[Cor.1]{PT}. In particular, for $\Psi$ coming 
from a complete flag, the constant term of $\mu^\Psi(K)$ gives Morelli's 
formula \cite{Mo}.
\item[(6)] (Inner Product) For $\Psi$ coming from an inner product $Q^*$ on 
$V^*$, the function $\mu^\Psi(K)$ agrees with $\mu(\bar{K})$ constructed in 
\cite{BV}, where $\bar K$ is the image of $K$ in $V/V(K)$, where $V(K)$ denotes the linear
subspace parallel to the largest affine subspace contained in $K$.
\end{itemize}
\end{theorem}
There are two corollaries of Theorem \ref{thm.1}: 
a local version of the Euler-MacLaurin formula for polytopes (promised
in the introduction), and a meromorphic function on $\Fl(V^*) \times V^*$ 
associated to a pointed cone $K$ in $V$, where $\Fl(V^*)$ is the 
variety of complete flags of $V$.

To formulate these results, fix a rigid complement map $\Psi$ on $V$, 
a $\Psi$-generic rational convex polyhedron $P$ and a face $F$ of $P$. 
Consider the Taylor series expansion of $\mu^{\Psi}(\Supp(P,F))(\xi)$ and the 
corresponding differential operator (of infinite order, with constant 
coefficients):

\begin{equation}
\lbl{eq.Dpf}
D^{\Psi}(P,F)=\mu^{\Psi}(\Supp(P,F))(\pt_x).
\end{equation}

\begin{theorem}
\lbl{thm.2}
Let $\Psi$ be a complement map on $V^*$, let $P$ be a $\Psi$-generic rational 
polytope in $V$, and let $h(x)$ be a polynomial function on $V$.
With the above notation,
we have
\begin{equation}
\lbl{eq.EM}
\sum_{x \in P \cap \La} h(x)=\sum_{F \in \calF(P)}\int_{F} 
D^{\Psi}(P,F) \cdot h.
\end{equation}
In particular, if $h(x)=1$ then
\begin{equation}
\lbl{eq.EML}
\#(P)=\sum_{F \in \calF(P)} \mu^{\Psi}(\Supp(P,F))(0) \vol(F)
\end{equation}
where $\#(P)$ (resp. $\vol(P)$) denotes the number of lattice points (resp. 
the volume) of $P$.
\end{theorem}
For concrete examples, see Section \ref{sec.examples}.
Equation \eqref{eq.EML} computes the number of lattice points of a polytope
as a weighted sum of the volume of its faces.

Suppose now that $\Psi$ comes from a flag $L$. Letting $L$ vary, we
obtain the following.

\begin{theorem}
\lbl{thm.3}
Given a pointed cone $K$ in $V$ there exists an effectively computable
meromorphic function $\mu(K)$ in $\Fl(V^*) \times V^*$ which is regular 
in a Zariski open subset of $\Fl(V^*) \times \{0\}$.
\end{theorem}
For an algorithmic computation of $\mu(K)$ for cones $K$ of dimension at 
most $2$, see Section \ref{sec.examples}. 

The next conjecture identifies the constant term of a $\Psi$-compatible
interpolator with the one of \cite{PT}. We post it as a conjecture for now,
and hope to discuss it in a future publication.

\begin{conjecture}
\lbl{conj.alldim}
Part (5) of Theorem \ref{thm.1} holds in all dimensions.
\end{conjecture}

\begin{question}
\lbl{que.1}
Is there any relation between the rational functions $\mu(K)$ in 
$\Fl(V^*) \times V^*$ and the quantum cohomology ring on $\Fl(V^*)$?
\end{question}

\subsection{\IS-interpolators and reverse Euler-MacLaurin formula}
\lbl{sub.IS}

The definition of the \SI-interpolator leads in a natural way to the
notion of an \IS-interpolator, which may be a useful notion to numerical
approximations of integrals by sums. The next notion
requires us to restrict attention to {\em lattice polyhedra}, i.e.,
polyhedra whose vertices are points of the lattice $\La$. 
Below, a {\em lattice cone} $K$ in $V$ is a rational cone such that 
$U \cap \La \neq \emptyset$, where $U$ is the largest affine subspace contained in $K$.

\begin{definition}
\lbl{def.ISinterpolator}
An {\em \IS-interpolator} is map $\l$, from the set of all 
 lattice cones in $V$ to $\calM(V^*)$  such that for any lattice
polyhedron $P$ in $V$ we have:
\begin{equation}
\lbl{eq.ISinter}
I(P)=\sum_{F \in \calF(P)} \l(\Supp(P,F))S(F)
\end{equation}
where the sum is over all faces $F$ of $P$.
An \IS-interpolator is {\em regular} if  $\l$ takes values in 
$\calM^r(V^*)$.
\end{definition}
Hereditary and $\Psi$-compatible \IS-interpolators are defined in an analogous
way to Definition \ref{def.interhereditary} and \ref{def.Psicompatible}.

The statement and proof of Theorem \ref{thm.1} holds with minor modification 
for \IS-interpolators. 

\begin{theorem}
\lbl{thm.4}
\rm{(a)}
If $V$ is a rational vector space and $\Psi$ is a rigid complement map on 
$V^*$, there is a unique $\Psi$-compatible \IS-interpolator $\l^{\Psi}$.
\newline
\rm{(b)} In addition, satisfies properties (2)-(5) of 
Theorem \ref{thm.1} and the following version of additivity:
\begin{enumerate}
\item
(Additivity) If $K$ is a cone that does not contain a linear
subspace and is subdivided into a finite union of
cones $K_i$, then $\l^{\Psi}(K)=\sum_{i:\text{dim}(K_i)=\text{dim}(K)}
\l^{\Psi}(K_i)$.
\end{enumerate}
\end{theorem}

One can obtain directly a formula for $\l^{\Psi}$ in terms of the 
interpolator $\mu^{\Psi}$ of Theorem \ref{thm.1} as follows, by observing
that $\mu^{\Psi}(\Supp(K,K))=1$.

\begin{theorem}
\lbl{thm.5}
If $K$ is a positive (resp. zero) dimensional lattice cone then
\begin{equation}
\lbl{eq.mobius}
\sum_{F \in \calF(K)} \l^{\Psi}(F) \mu^{\Psi}(\Supp(K,F))=0,
\qquad \text{(resp.)} \qquad \l^{\Psi}(K) =1.
\end{equation}
\end{theorem}
The interpolator $\l^{\Psi}$ gives a reverse Euler-MacLaurin summation
formula. Fix a rigid complement map $\Psi$ on $V$, 
a $\Psi$-generic lattice polyhedron $P$ and a face $F$ of $P$. 
Consider the Taylor series expansion of $\l^{\Psi}(\Supp(P,F))(\xi)$ and 
the corresponding differential operator (of infinite order, with constant 
coefficients):

\begin{equation}
\lbl{eq.ISDpf}
\Delta^{\Psi}(P,F)=\l^{\Psi}(\Supp(P,F))(\pt_x)
\end{equation}

\begin{theorem}
\lbl{thm.22}
With the above assumptions, for every lattice polytope $P$ and every
polynomial function $h(x)$ on $V$ we have
\begin{equation}
\lbl{eq.EM22}
\int_P  h(x)=\sum_{F \in \calF(P)} \sum_{x \in F \cap \La}
\left(\Delta^{\Psi}(P,F) \cdot h\right)(x).
\end{equation}
In particular, if $h(x)=1$ then
\begin{equation}
\lbl{eq.EML2}
\vol(P)=\sum_{F \in \calF(P)} \l^{\Psi}(\Supp(P,F))(0) \#(F)
\end{equation}
where $\#(F)$ denotes the number of lattice points of $F$.
\end{theorem}
Equation \eqref{eq.EML2} computes the volume of a lattice polytope 
as a weighted sum of the number of lattice points of its faces. A formula
of this type was first written down by Morelli, \cite[Eqn.(5)]{Mo}.

\subsection{\ISo-interpolators and Morelli's work}
\lbl{sub.morelli}

When $\Psi$ comes from a complete flag, Equations \eqref{eq.EML} and 
\eqref{eq.EML2} are similar with some results \cite[Eqn.(5),Eqn.(6)]{Mo} 
of Morelli. To explain this, let us introduce the variant $S^0$ of the 
exponential sum function defined by:

\begin{equation}
\lbl{eq.S0}
S^0(P)=S(P^0)
\end{equation}
where $P^0$ denotes the (relative) interior of a polyhedron $P$. The set $P^0$ is not a 
polyhedron itself, however it is a virtual sum of polyhedra.
Then, we can talk about $\IS^0$-interpolators $\nu$ on {\em lattice} cones. 
Theorems \ref{thm.4} and \ref{thm.5} have the following analogue.

\begin{theorem}
\lbl{thm.6}
\rm{(a)}
If $V$ is a rational vector space and $\Psi$ is a rigid complement map on 
$V^*$, there is a unique $\Psi$-compatible \ISo-interpolator 
$\nu^{\Psi}$. In addition, 
\newline
\rm{(b)} (Additivity) If $K$ is a cone that does not contain a linear
subspace and is subdivided into a finite union of
cones $K_i$, then $\nu^{\Psi}(K)=\sum_{i:\text{dim}(K_i)=\text{dim}(K)}
\nu^{\Psi}(K_i)$.
\newline
\rm{(c)} $\nu^{\Psi}$ satisfies properties (2)-(5) of Theorem \ref{thm.1}.
\newline
\rm{(d)} If $K$ is a lattice cone then
\begin{equation}
\lbl{eq.mobiusIS0}
\sum_{F \in \calF(K)} \nu^{\Psi}(F) \mu^{\Psi}(\Supp(K,F))=1.
\end{equation}
\end{theorem}
Fix a lattice polytope $P$  and a face $F$ of $P$. If 
$\Psi$ comes from a complete flag $L$ in $V^*$ as in Lemma 
\ref{lem.examplesPsi}, and we vary the flag, we obtain two rational 
functions $\nu(\Supp(P,F))$ and 
$\mu(\Supp(P,F))$ on $\Fl(V^*) \times V^*$ which are regular at 
a Zariski open subset of $\Fl(V^*) \times 0$.
Thus, we can consider the constant terms $\nu^{\Psi}(\Supp(P,F))(0)$ and 
$\mu^{\Psi}(\Supp(P,F))(0)$ which are rational functions on $\Fl(V^*)$. 

On the other hand, Morelli constructs a pair of rational functions 
$\nu_k(\vartheta_F P)$ and $\mu_k(\vartheta_F P)$ (where $k$ is the dimension 
of $\lin(\Supp(P,F))$) that appear respectively in Equations (5) and (6) of 
\cite{Mo} (unfortunately, Morelli denotes the two functions with the same
notation, although he clearly recognizes that they are distinct functions). 
Morelli's construction uses K-theory, localization and the Bott residue 
theorem.

The following proposition is a re-expression of the flag case of Theorem 
\ref{thm.1}, part (5).
\begin{proposition}
\lbl{prop.mor1}
With the above notations, for polytopes $P$ of dimension at most 2, we have:
\begin{equation}
\lbl{eq.munu}
\nu(\Supp(P,F))(0)=\nu_k(\vartheta_F P), \qquad
\mu(\Supp(P,F))(0)=\mu_k(\vartheta_F P).
\end{equation}
\end{proposition}

In \cite[p.191]{Mo} and also in \cite[Thm.5]{Mo}, Morelli observes that 

\begin{equation}
\lbl{eq.mu=nu}
\mu(K)(0)=\nu(K^{\vee})(0)
\end{equation}
for lattice cones $K$ of dimension 4 or less, and that Equation
\eqref{eq.mu=nu} fails in dimensions more than 4.
This mysterious and seemingly deep coincidence is easy to explain from our point of 
view, for cones of dimension at most 2. The rational functions $\nu(K)$ 
and $\mu(K^{\vee})$ are distinct, 
in fact they are functions on distinct spaces ($\Fl(V^*) \times V^*$ and 
$\Fl(V) \times V$, respectively).
 Their constant terms coincidentally agree in small dimensions.
For a discussion, see Section \ref{sub.mu=nu}.

\section{Proofs}
\lbl{sec.proofs}

\subsection{Proof of Theorem \ref{thm.1}}
\lbl{sub.thm1}

Fix a rigid complement map $\Psi$ on a rational vector space $V$ with
lattice $\La$, and consider a $\Psi$-compatible interpolator $\mu^{\Psi}$
and a $\Psi$-generic rational polytope $P$. We will show by induction 
on the dimension of $P$ the
existence and uniqueness of $\mu^{\Psi}(P)$. All cones in this section 
will be assumed to be $\Psi$-generic. By Lemma \ref{lem.interpolator} 
it suffices to assume that $P=K$ is a cone in $V$. To avoid confusion,
we denote the map $\mu: \calC^{\Psi}(V) \longto \calM(V^*)$ by $\mu_V$
and for every $\Psi$-generic quotient $W$ of $V$, we denote the map
$\bar{\mu}:\calC^{\bar\Psi}(W) \longto \calM(W^*)$ by $\mu_W$.

If $K$ is pointed with vertex $v$, then we can single out the contribution 
from the 
$0$-dimensional face of $K$ in Equation \eqref{eq.inter}, and together with 
hereditary property of Equation \eqref{eq.comp1}, it follows that

\begin{equation}
\lbl{eq.thm1a}
\mu^{\Psi}_V(K)(\xi)=e^{-\la \xi,v\ra}\left(\sum_{x \in K \cap \La} 
e^{\la \xi,x\ra}-\sum_{F, \text{dim}(F)>0} 
\mu^{\bar\Psi}_{V/\lin(F)}(\overline{\Supp(K,F)})(\pi(\xi)) \int_F e^{\la \xi,x\ra} dm_F
(x) \right)
\end{equation}
where $\overline{\Supp(K,F)}$ denotes the image of $\Supp(P,F)$ in the quotient
space $V/\lin(F)$, and  $\pi: V^* \longto (V/\lin(F))^*$ is the projection 
using the decomposition $V^*=(V/\lin(F))^* \oplus \Psi((V/\lin(F))^*)$
given by the rigid complement map $\Psi$

On the other hand, if $K$ is a non-pointed cone in $V$, let $V(K)$ denote the linear subspace
of $V$ parallel to the largest affine subspace contained in $K$, and consider the
image $\bar K$ of $K$ in $V/V(K)$, which is a pointed cone of dimension
strictly less than the dimension of $K$.
Equations \eqref{eq.comp1} and \eqref{eq.comp2} imply that 
\begin{equation}
\lbl{eq.thm1b}
\mu^{\Psi}_V(K)(\xi)=\mu^{\bar\Psi}_{V/V(K)}(\bar K)(\bar \xi)
\end{equation}
where $\xi \in V^*$, $\pi: V^*  \longto (V/V(K))^*$ is the projection 
using the decomposition $V^*=(V/V(K))^* \oplus \Psi((V/V(K))^*)$
given by the rigid complement map $\Psi$ and 
$\bar \xi \in (V/V(K))^*$ is the image of $\xi$ under the above projection.

Equations \eqref{eq.thm1a} and \eqref{eq.thm1b} uniquely define $\mu^{\Psi}_V$ 
from $\mu^{\Psi}_W$ for $\text{dim}(W)< \text{dim}(V)$. On the other hand,
when $V$ is a $0$-dimensional space, we have $\mu^{\Psi}_V(\{0\})=1$. 
This uniquely determines $\mu^{\Psi}_V$. An explicit computation of 
$\mu^{\Psi}(K)$ for cones $K$ of dimension $0$, $1$ and $2$ is given in 
Section \ref{sec.examples}.

Additivity of $\mu^{\Psi}$ follows from the above inductive definition of 
$\mu^{\Psi}$ and the additivity (property (A2)) of the exponential sum and 
exponential integrals. This is presented in detail in the proof of 
\cite[Prop.15]{BV}. 

\begin{remark}
\lbl{rem.TPF}
In fact, Berline-Vergne in \cite{BV} work with the {\em transverse cone} 
$\T(K,F)$ of a face $F$ of $K$ defined to be the image of $\Supp(K,F)$ under 
the projection map $V \longto V/\lin(F)$. The transverse cone $\T(K,F)$ is always
a pointed cone. In fact this gives a 1-1 correspondence

\begin{equation}
\lbl{eq.V2V}
\{ \text{cones in } V \} \leftrightarrow
\{ \text{pointed cones in quotients of } V \}
\end{equation}
Using this correspondence, and an inner product $Q$ on $V$,
Berline-Vergne inductively construct $\mu^{Q}$ defined on the set of
pointed cones on the quotients of $V$. For visual reasons, we prefer
to work with the supporting cones $\Supp(K,F)$ rather than the transverse
cones $\T(K,F)$.
\end{remark}

Going back to the proof of Theorem \ref{thm.1}, 
the lattice invariance and the isometry equivariance of $\mu^{\Psi}$ follows 
easily by induction and the invariance of the $I$ and $S$ functions.

The regularity of the meromorphic function $\mu^{\Psi}(K)$ at zero follows 
by induction and a residue calculation, discussed in detail in 
\cite[Prop.18]{BV}, using the following 
fact about the exponential sum and integral functions.

\begin{itemize}
\item[(A6)] If $K$ is a cone in $V$ with primitive integral generators
$v_1,\dots,v_k$, then  
\begin{equation}
\lbl{eq.A5a}
\prod_{i=1}^k \la \xi,v_i\ra S(K)(\xi),
\qquad
\prod_{i=1}^k \la \xi,v_i\ra I(K)(\xi)
\end{equation}
are regular functions and the residues of the meromorphic functions $S(K)$ 
and $I(K)$ along the hyperplane $v_1=0$ satisfy the equations
\begin{equation}
\lbl{eq.A5b}
\Res_{v_1}(S(K))=-S(\bar K), \qquad
\Res_{v_1}(I(K))=-I(\bar K).
\end{equation}
\end{itemize}
(A6) follows in turn by additivity and Equation \eqref{eq.A5} which evaluates
the $I$ and $S$ functions on simplicial cones.

We  defer the proof of Property (5) of Theorem \ref{thm.1}  until the 
example section (Section \ref{sec.examples}).
See the proofs immediately following Examples \ref{ex.1} and \ref{ex.2}.

Finally property (6) of Theorem \ref{thm.1} follows directly from
\cite[En.4]{BV}.

\subsection{Proof of Theorem \ref{thm.2}}
\lbl{sub.thm2}

Fix a polynomial function $h(x)$ on $V$ and a polytope $P$ on $V$.
Consider the constant coefficient differential operator $h(\pt_\xi)$ which
acts on the function $\xi \mapsto e^{\la \xi, x\ra}$ by:

\begin{equation}
\lbl{eq.Dacts}
 h(\pt_\xi) e^{\la \xi, x\ra} = h(x) e^{\la \xi, x\ra}.
\end{equation}
The definition \eqref{eq.A4} of $S(P)$ and the above gives that

$$
h(\pt_\xi) S(P)=\sum_{x \in P \cap \La} h(x) e^{\la \xi, x\ra}
$$
Let $\text{ev}$ denote the evaluation at $\xi=0$. It follows that

\begin{equation}
\lbl{eq.thm2a}
(\text{ev} \circ h(\pt_\xi)) S(P)
=\sum_{x \in P \cap \La} h(x).
\end{equation}
On the other hand, Theorem \ref{thm.1} gives that:

\begin{eqnarray*}
S(P)(\xi) &=& \sum_{F \in \calF(P)} \mu^{\Psi}(\T(P,F))(\xi) I(F)(\xi)\\
&=& \sum_{F \in \calF(P)} \int_F \mu^{\Psi}(\T(P,F))(\xi) 
e^{\la \xi, x\ra} dm_P(x) \\
&=& \sum_{F \in \calF(P)} \int_F \mu^{\Psi}(\T(P,F))(\pt_x) 
e^{\la \xi, x\ra} dm_P(x) \\
&=& \sum_{F \in \calF(P)} \int_F D^{\Psi}(P,F) \,
e^{\la \xi, x\ra} dm_P(x)
\end{eqnarray*}
where the last equality follows from the definition \eqref{eq.Dpf} of
the differential operator $D^{\Psi}(P,F)$. Applying the differential
operator $h(\pt_{\xi})$ to both sides it follows that

$$
h(\pt_\xi) I(P)(\xi)=\sum_{F \in \calF(P)} \int_F D^{\Psi}(P,F) \, h(x)
e^{\la \xi, x\ra} dm_P(x).
$$
Evaluating at $\xi=0$, it follows that

\begin{equation}
\lbl{eq.thm2b}
(\text{ev} \circ h(\pt_\xi)) S(P)
= \sum_{F \in \calF(P)} \int_F D^{\Psi}(P,F) \cdot h.
\end{equation}
Equations \eqref{eq.thm2a} and \eqref{eq.thm2b} complete 
the proof of Theorem \ref{thm.2}.

\subsection{Proof of Theorem \ref{thm.4}}
\lbl{sub.thm4}

The proof of Theorem \ref{thm.1} applies verbatim, with the following 
observation, which explains the need for lattice polyhedra.

If $K$ is a rational pointed cone with vertex $v$ for which \eqref{eq.ISinter} 
applies, then one of the terms in \eqref{eq.ISinter} is $F=\{v\}$.
In that case, $\T(K,\{v\})=K$ and the corresponding term in 
\eqref{eq.ISinter} is $\nu(K)S(\{v\})$. Now, (A3) and (A5) imply
that

$$
S(\{v\})(\xi)=\begin{cases} e^{\la \xi, v \ra} & \text{if} \, v \in \La \\
0 & \text{otherwise.}
\end{cases}
$$ 
Thus, when $K$ is not a lattice cone, then we cannot solve for $\nu(K)$.
In the case of \SI-interpolators, as discussed in Section 
\ref{sub.cmap}, the corresponding term of \eqref{eq.inter} was 
$\mu(K)I(\{v\})$ and (A3) and (A5) imply that

$$
I(\{v\})(\xi)= e^{\la \xi, v \ra}.
$$

\subsection{Proof of Theorems \ref{thm.5} and \ref{thm.6}}
\lbl{sub.thm5}

Theorem \ref{thm.5} follows from the fact that the composition of an 
\SI-interpolator with an \IS-interpolator is an \rm{SS}-interpolator,
and the fact that an \rm{SS}-interpolator is unique. 
More precisely, fix an \SI-interpolator $\mu$ and an \IS-interpolator $\l$, 
and consider a 
lattice cone $K$. Then, on the one hand we have
$$
S(K)=\sum_{F' \in \calF(K)} \mu(\Supp(K,F'))I(F').
$$
On the other hand for every face $F'$ of $K$, we have
$$
I(F')=\sum_{F \in \calF(F')} \l(\Supp(F',F))S(F)
$$
Substituting the second equation into the first, it follows that
$$
S(K)=\sum_{F, F': \, F \subset F'} 
\mu(\Supp(K,F')) \l(\Supp(F',F)) S(F).
$$
This motivates the following definition. Consider the function
$$
\Supp(P,F) \mapsto (\l \circ \mu) (\Supp(P,F)):=\sum_{F': \,  F \subset F'} 
\mu(\Supp(K,F')) \l(\Supp(F',F)).
$$
It is easy to see that $\l \circ \mu$ is an $SS$-interpolator. Moreover,
if $\mu$ and $\l$ are $\Psi$-compatible, so is their composition.
On the other hand, there is a unique $\Psi$-compatible $SS$-interpolator.
Thus $(\l \circ \mu)(\Supp(P,F))=0 \, \text{resp.} \, 1$ 
for $F \neq P$ (resp. $F=P$). This concludes the proof of Theorem 
\ref{thm.5}. 

Theorem \ref{thm.6} follows by an analogous computation, using the fact that
the map $(\nu \circ \mu)$ is a $\Psi$-compatible $SS^0$-intepolator, 
the uniqueness of such interpolators, 
and the fact that the map that sends every cone to $1$ is a 
$\Psi$-compatible $SS^0$-intepolator.

\section{Computations and Examples}
\lbl{sec.examples}

\subsection{Computation of $\mu^{\Psi}$ for low dimensional cones}
\lbl{sub.computemu}

As was mentioned earlier, the meromorphic functions $\mu^{\Psi}(K)$ are 
effectively computable given a complement map $\Psi$. In this section
we illustrate this, by explicitly computing $\mu^{\Psi}(K)$ for cones $K$ of
dimension at most $2$.

\begin{proposition}
\lbl{prop.mu0}
Suppose that $K=\{0\}$ in the vector space  $V=\{0\}$.  Then  $\mu(K)$ is 
the constant function 1, independent of $\Psi$.  
\end{proposition}

\begin{proof}
$S(K)=I(K)=1$ forces $\mu(K)=1$.
\end{proof}

If $\{v_1,\dots,v_k\}$ is a collection of vectors in $V$, let 
$
(v_1,\dots,v_k)=\BQ^+ v_1+ \dots + \BQ^+ v_k
$
denote the cone spanned by those vectors, where $\BQ^+$ is the set of 
non-negative rational numbers.

\begin{proposition} 
\lbl{prop.mu1}
Now suppose that $V$ is one-dimensional and  $K= \Cone(v)$ is a ray in $V$ 
generated by a primitive vector $v\in\Lambda$. Then, independent of $\Psi$,
 we have
$$
\mu(K)(\xi)=B(\langle \xi, v \rangle),
$$
where 
\begin{equation}
\lbl{eq.Ber}
B(z)=\frac{1}{1-e^z} + \frac{1}{z}=\frac{1}{2}-\frac{z}{12}+\frac{z^3}{720}
-\frac{z^5}{30240}+ \dots
\end{equation}
is the generating series of the Bernoulli numbers.
\end{proposition}

\begin{proof}
Use the property of interpolators
$$
S(K) = \mu(\Supp(K,K)) I(K) + \mu(\Supp(K,0))I(0),
$$
 $\mu(\Supp(K,K))=1$ by the previous proposition and compatibility under quotients. In addition,
we have $S(K)(\xi)=1/(1-e^{\xi})$ and $I(K)(\xi)=-1/\xi$. The result follows.
\end{proof}

For the following proposition, recall that if $K$ is a subset of $V$, then we define the
dual $K^{\vee}$ by

\begin{equation}
\lbl{eq.checkS}
K^{\vee}=\{w \in V^* \, | \, \la w,v \ra \geq 0, \,\, \text{for all} \,\,
v \in K\}.
\end{equation}

Also, if $\La$ is a lattice in $V$, then $\La^*=\{w \in V^* \, | \la w,v \ra
\in \BZ, \,\, \text{for all} \,\, v \in \La \}$ is the dual lattice in
$V^*$.
Below, we will denote by $\mu^L$ (resp. $\mu^Q$) the $\Psi$-compatible 
\SI-interpolator where $\Psi$ comes from a complete flag $L$ on $V$ (resp.
an inner product $Q$ on $V$), as in Lemma \ref{lem.examplesPsi}.

\begin{proposition}
\lbl{prop.mu2line}
Suppose $V$ is two-dimensional, and  $K\subset V$ is a half-plane with 
boundary a line $U$ through the origin.  
$$
\psdraw{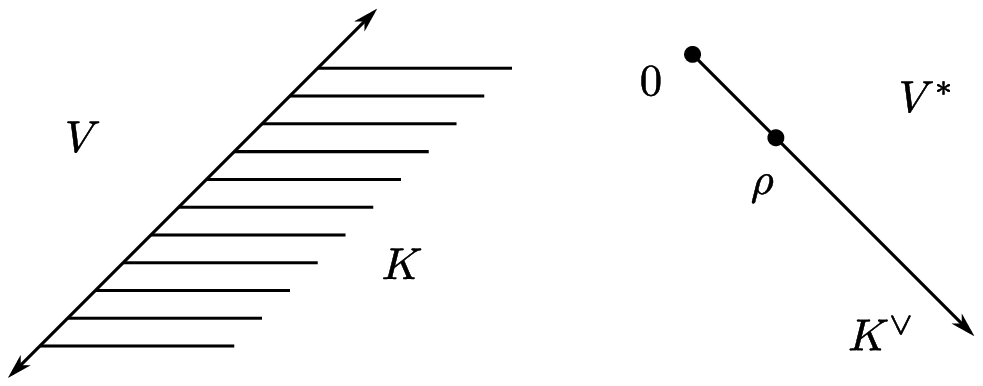}{3in}
$$
Let $\rho\in \Lambda^*$ be the primitive generator of the ray 
$K^{\vee}$. Suppose the $\Psi$ is defined by a complete flag $L$ in $V^*$, 
and let $c$ be a generator of $L_1^*$, where $L_1$ is 
the one-dimensional subspace of the 
flag $L$ (the only nontrivial subspace in the flag.) Then 
$$
\mu^L(K)(\xi) = B\biggl(  
\frac{\langle \xi, c \rangle}{\langle \rho, c \rangle}   \biggr).
$$
Now suppose $\Psi$ is defined by an inner product $Q$ on $V$.  Then 
$$
\mu^Q(K)(\xi) = B\biggl(  \frac{Q( \xi, \rho )}{Q( \rho, \rho )}   
\biggr).
$$
\end{proposition}

\begin{proof}
We use the $\Psi$-compatibility  together with our one-dimensional 
formula above.  To do so, we must compute the projection from 
$\pi : V^* \rightarrow W^*$ where $W=V/U$. This projection is computed 
using the decomposition $V^*= W^*\oplus \Psi^*(W^*)$ given by the 
complement map $\Psi$.   Suppose the 1-dimensional subspace 
$\Psi(W^*)\subset V^*$ is generated by $d$.  Then for $\xi\in V^*$, 
find scalars $\omega_1, \omega_2$ such that 
\begin{equation}
\lbl{eq.2ddecomp}
\xi = \omega_1 \rho  + \omega_2 d,
\end{equation}
Then from $\Psi$-compatibility and the one-dimensional 
formula, we obtain
$$
\mu(K)(\xi) = B(\omega_1).
$$
For $\Psi$ coming from a complete flag, we pair both sides of Equation 
\eqref{eq.2ddecomp} with the generator $c$.  Since $\langle d, c \rangle = 0$, 
we find
$$ 
\omega_1 =  \frac{\langle \xi, c \rangle}{\langle \rho, c \rangle},
$$
and the formula for $\mu^L(K)$ follows. For $\Psi$ coming from the 
inner product $Q$, we note that $Q(d,\rho) = 0$. Thus Equation 
\eqref{eq.2ddecomp} yields
$$
\omega_1 =  \frac{Q( \xi, \rho )}{Q( \rho, \rho )},
$$ 
which completes the proof.
\end{proof}

Using the defining property of interpolators, we immediately get the 
following proposition, an explicit expression for $\mu$ of a nonsingular 
$2$-dimensional cone.

\begin{proposition}
\lbl{prop.mu2}
Suppose $V$ is two-dimensional and 
$K=\Cone(v_1, v_2)$ where $v_1, v_2$ form a basis of 
$\Lambda\subset V$.  Let $F_1=\Cone( v_1 )$ and 
$F_2=\Cone( v_2 )$, and let $\rho_1, \rho_2\in \Lambda^*$ denote the 
primitive normals to $F_1$ and $F_2$.  
Then for $\Psi$ coming from a complete flag $L$ in $V^*$, we have
\begin{equation}
\lbl{eq.mu2da}
\mu^L(K)= 
\frac1{(1-e^{\langle \xi, v_1 \rangle})(1-e^{\langle \xi, v_2\rangle})}
- \frac{1}{\langle \xi, v_1 \rangle \langle \xi, v_2 \rangle}
+ \frac1{\langle \xi, v_1 \rangle} B\biggl(  
\frac{\langle \xi, c \rangle}{\langle \rho_1, c \rangle}   \biggr)
+ \frac1{\langle \xi, v_2 \rangle} B\biggl(  
\frac{\langle \xi, c \rangle}{\langle \rho_2, c \rangle}   \biggr).
\end{equation}
For $\Psi$ coming from an inner product $Q$ on $V^*$, we have
\begin{equation}
\lbl{eq.mu2db}
\mu^Q(K)= 
\frac1{(1-e^{\langle \xi, v_1 \rangle})(1-e^{\langle \xi, v_2\rangle})}
- \frac{1}{\langle \xi, v_1 \rangle \langle \xi, v_2 \rangle}
+ \frac1{\langle \xi, v_1 \rangle} B\biggl(  
\frac{Q( \xi, \rho_1 )}{Q( \rho_1, \rho_1 )}   \biggr)
+ \frac1{\langle \xi, v_2 \rangle} B\biggl(  
\frac{Q( \xi, \rho_2 )}{Q( \rho_2, \rho_2 )}   \biggr).
\end{equation}
\end{proposition}
 
\begin{remark}
\lbl{rem.muB}
Replacing $1/(1-e^x)$ with $B(x)-1/x$, it follows that Equations
\eqref{eq.mu2da} and \eqref{eq.mu2db} can be written in the form:
\begin{eqnarray*}
\mu^L(K)&=& B(\langle \xi, v_1 \rangle)B(\langle \xi, v_2 \rangle)
+ \frac1{\langle \xi, v_1 \rangle} \left(B\biggl(  
\frac{\langle \xi, c \rangle}{\langle \rho_1, c \rangle}   \biggr)
-B(\langle \xi, v_2 \rangle)\right) 
+ \frac1{\langle \xi, v_2 \rangle} \left(B\biggl(  
\frac{\langle \xi, c \rangle}{\langle \rho_2, c \rangle}   \biggr)
-B(\langle \xi, v_1 \rangle)\right) \\
\mu^Q(K)&=& B(\langle \xi, v_1 \rangle)B(\langle \xi, v_2 \rangle)
+ \frac1{\langle \xi, v_1 \rangle} \left(B\biggl(  
\frac{Q( \xi, \rho_1 )}{Q( \rho_1, \rho_1 )} \biggr)
-B(\langle \xi, v_2 \rangle)\right) \\
& & 
+ \frac1{\langle \xi, v_2 \rangle} \left(B\biggl(  
\frac{Q( \xi, \rho_2 )}{Q( \rho_2, \rho_2 )}   \biggr)
-B(\langle \xi, v_1 \rangle)\right).
\end{eqnarray*}
\end{remark}

\subsection{Examples}
\lbl{sub.examples}

In this section we explicitly compute some low dimensional examples
and match them with the ones given by Morelli, Pommersheim-Thomas and 
Berline-Vergne.
 
\begin{example}
\lbl{ex.1} 
Consider the triangle $P$ with vertices $v_0=(0,0),v_1=(1,0),v_2=(0,1)$ 
in the lattice $\Lambda= \BZ^2 \subset V = \BQ^2$, where $V^*$ is also 
identified with $\BZ^2$.
$$
\psdraw{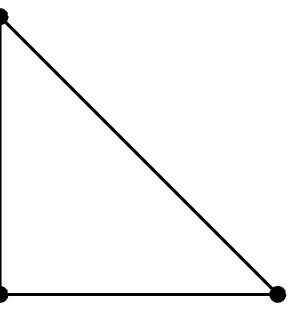}{0.5in}
$$
Let $\Psi$ be the complement map coming from the complete flag 
$L=(L_0,L_1,L_2)$ in $V^*$ such that $L_1= \Span\{(d_1,d_2)\}$.
To compute $\mu^L$ of the supporting cones, we use lattice invariance
to translate them to the origin. In other words, consider the
tangent cones 
$K_i=\Tan(P,v_i)$ for $i=0,1,2,$ to $P$ at $v_i$. 
Specifically we have
$$
K_0=\Cone((1,0), (0,1)), \qquad K_1=\Cone((-1,1), (-1,0)), \qquad 
K_2=\Cone((0,-1), (1,-1)).
$$  
Then for $\xi = (\xi_1, \xi_2)\in V^*$, we have

\begin{align*}
\mu^L(K_0)(\xi) = &\frac1{(1-e^{\xi_1})(1-e^{\xi_2})}
- \frac{1}{\xi_1 \xi_2}
+ \frac1{\xi_2} B\biggl(  \xi_1 -\frac{d_1}{d_2}\xi_2   \biggr)
+ \frac1{\xi_1} B\biggl(  \xi_2 -\frac{d_2}{d_1}\xi_1   \biggr) \\
\mu^L(K_1)(\xi) = &\frac1{(1-e^{\xi_2-\xi_1})(1-e^{-\xi_1})}
+ \frac{1}{\xi_1 (\xi_2 -\xi_1)}
- \frac1{\xi_1} B\biggl(  \xi_2 -\frac{d_2}{d_1}\xi_1   \biggr)
+ \frac1{\xi_2-\xi_1} B\biggl( \frac{d_2\xi_1 - d_1\xi_2}{d_1-d_2}   \biggr) \\
\mu^L(K_2)(\xi) = &\frac1{(1-e^{\xi_1-\xi_2})(1-e^{-\xi_2})}
+ \frac{1}{\xi_2 (\xi_1 -\xi_2)}
- \frac1{\xi_2} B\biggl(  \xi_1 -\frac{d_1}{d_2}\xi_2   \biggr)
+ \frac1{\xi_1-\xi_2} B\biggl( \frac{d_2\xi_1 - d_1\xi_2}{d_1-d_2}   \biggr)
\end{align*}
These functions are analytic at $\xi=(0,0)$. For example, 
$\mu^L(K_0)(\xi)$ has the Taylor expansion
\begin{align*}
\mu^L(K_0)(\xi)=\biggl(    \frac{d_1^2 + d_2^2 + 3 d_1 d_2}{12 d_1 d_2}  
\biggr) &+
\biggl(  \frac{-1}{24}   \biggr) \xi_1 +
\biggl(  \frac{-1}{24}   \biggr) \xi_2 +
\biggl(  \frac{3d_1^4+3d_2^4+5d_1^2d_2^2}{720 d_1^2 d_2^2}  \biggr) 
\xi_1 \xi_2\\
&+ \biggl( \frac{-3d_1^4-d_2^4}{720d_1^3d_2}  \biggr) \xi_1^2+
\biggl(   \frac{-3d_2^4-d_1^4}{720d_2^3d_1}  \biggr) \xi_2^2+\cdots
\end{align*}

Denoting the constant term of each $\mu^L(K_i)$ by $\mu_0^L(K_i)$, 
we find

\begin{align*}
\mu_0^L(K_0) &=   \frac{d_1^2 + d_2^2 + 3 d_1 d_2}{12 d_1 d_2}\\
\mu_0^L(K_1) &=   \frac{5d_1^2-5d_1 d_2+d_2^2}{12d_1(d_1-d_2)}\\
\mu_0^L(K_2) &=   \frac{5d_2^2-5d_1 d_2+d_1^2}{12d_2(d_2-d_1)}
\end{align*}
in exact agreement with Morelli's example \cite[p.198-199]{Mo}.  These rational 
functions can also be
computed by multiplying the Todd polynomials using \cite[Theorem 3]{PT}
In particular, these three rational functions sum to 1 in agreement with
Equation \eqref{eq.EML}.

In addition, one checks by direct computation that
$$
\sum_{F} \mu^L(\Supp(P,F)) I(F) = S(P) = 1+ e^{\xi_1}+ e^{\xi_2}.
$$
in agreement with the definition of an interpolator.
\end{example}

Indeed, this example provides a general proof of the coincidence of our
construction and the
flag case of \cite{PT} for cones of dimension at most 2, and hence the 
coincidence with Morelli's formulas.

{\sl Proof of \ref{thm.1}, part (5) for flags:}
In dimension 1, the \cite{PT} values are always equal to ${1 \over 2}$, 
independent of complement map.
 Using Proposition \ref{prop.mu1}, we see that this coincides with the 
constant term of $\mu$.
For the two-dimensional case, we first use the additivity of \cite{PT} 
and of our construction (Theorem 1, Property (1)) to reduce
to nonsingular $2$-dimensional cones.  By functoriality, we may assume that 
$K$ is the cone $K_0$ from example
\ref{ex.1}, generated by the standard basis of $\BZ^2$.  In this case, we 
note that  the values of $\mu_0^L(K_0)$ given in Example
 \ref{ex.1} match the values in  \cite[p.198-199]{Mo}.  By \cite[Corollary 2]{PT}, these values
 also match the values in \cite{PT}.
This completes the proof.

\begin{example}
\lbl{ex.2}
We consider the same triangle $P$ from Example \ref{ex.1}, but now with a 
complement map $\Psi$ coming from an inner product $Q$ on $V^*$, given 
by the matrix 
$$
\left(\begin{matrix}
a & b \\
b & c
\end{matrix}\right)
$$
We can use Proposition \ref{prop.mu2} to compute $\mu^Q(K_i)$ for the 
supporting cones $K_i$ of $P$. We find
\begin{align*}
\mu^Q(K_0)(\xi) &= \frac1{(1-e^{\xi_1})(1-e^{\xi_2})}
- \frac{1}{\xi_1 \xi_2}
+ \frac1{\xi_2} B\biggl(  \xi_1 +\frac{b}{a}\xi_2   \biggr)
+ \frac1{\xi_1} B\biggl(  \xi_2 +\frac{b}{c}\xi_1   \biggr) \\
\mu^Q(K_1)(\xi)&= \frac1{(1-e^{\xi_2-\xi_1})(1-e^{-\xi_1})}
+ \frac{1}{\xi_1 (\xi_2 -\xi_1)}
- \frac1{\xi_1} B\biggl(   \xi_2 +\frac{b}{c}\xi_1   \biggr)
+ \frac1{\xi_2-\xi_1} B\biggl( \frac{-a\xi_1-b\xi_1-b\xi_2-c\xi_2}{a+2b+c}   
\biggr) \\
\mu^Q(K_2)(\xi)&= \frac1{(1-e^{\xi_1-\xi_2})(1-e^{-\xi_2})}
+ \frac{1}{\xi_2 (\xi_1 -\xi_2)}
- \frac1{\xi_2} B\biggl(   \xi_1 +\frac{b}{a}\xi_2   \biggr)
+ \frac1{\xi_1-\xi_2} B\biggl( \frac{-a\xi_1-b\xi_1-b\xi_2-c\xi_2}{a+2b+c}   
\biggr)
\end{align*}

This time, we find the constant terms are given by
\begin{align*}
\mu_0^Q(K_0) &=   \frac{3ac-ab-bc}{12ac}\\
\mu_0^Q(K_1) &=   \frac{ab+4ac+10bc+2b^2+5c^2}{12(ac+2bc+c^2)}\\
\mu_0^Q(K_2) &=   \frac{5a^2+2b^2+10ab+4ac+bc}{12(a^2+2ab+ac)}
\end{align*}

These agree with the values of $\mu$ constructed in \cite[Cor.1]{PT} for 
$\Psi$ coming from an inner product.  In particular, the three rational 
functions above sum to 1.   For the standard inner product on $V^* = \BZ^2$,  
(corresponding to $a=c=1$, $b=0$), we find 
$$
\mu_0^Q(K_0)= \frac{1}{4} \qquad \mu_0^Q(K_1) = \frac{3}{8} 
\qquad  \mu_0^Q(K_2)=\frac{3}{8}.
$$

For completeness and comparison, we detail how these numbers arise, in a 
seemingly very different way, out of the construction of Pommersheim-Thomas.
We begin, according to \cite[Cor.1(iv)]{PT} by considering the outer 
normal fan of $P$.  The rays of this fan
are $\rho_0=(1,1)$, $\rho_1=(-1,0)$, and $\rho_2=(0,-1)$.   We must then 
multiply the second-degree Todd polynomial $T=\sum_{i<j} D_i D_j  + 
\frac{1}{12} \sum_i D_i^2$ in the ring given in \cite[Prop.2]{PT}.  
In this ring, one finds $D_0^2= \frac{1}{2}D_0D_1+\frac{1}{2}D_0D_2$, 
$D_1^2 = D_0D_1$, and $D_2^2= D_0D_2$.  Hence
$$
T= \frac{1}{4} D_1D_2 + \frac{3}{8} D_0D_2 + \frac{3}{8} D_0D_1,
$$
from which one reads off the values of $\mu_0(K_0)=\frac{1}{4}$, 
$\mu_0(K_1)=\frac{3}{8}$, and $\mu_0(K_2)=\frac{3}{8}$. In fact, more 
generally, let $K=\cone(v_1,v_2)$ be any nonsingular cone in $V$ with 
vertex at $0$, and let $Q$ be an inner product on $V^*$.  Then an easy 
computation similar to the above can be used to compute the value of 
$\mu_0$ of \cite{PT} for this cone $K$. Indeed, if  
$K^*=\cone(\rho_1, \rho_2)$ is the dual cone, then using presentation of 
\cite[Proposition 2]{PT}, one finds that the coefficient of $D_1D_2$ in 
$T$ is given by 

\begin{equation}
\lbl{eq.ptip2d}
\frac{1}{4}-\frac{Q(\rho_1,\rho_2)}{12}\biggl( \frac{1}{Q(\rho_1, \rho_1)} 
+ \frac{1}{Q(\rho_2, \rho_2)}  \biggr)
=
\frac{1}{4}+\frac{Q^*(v_1,v_2)}{12}\biggl( \frac{1}{Q^*(v_1, v_1)} 
+ \frac{1}{Q^*(v_2, v_2)}  \biggr),
\end{equation}
where $Q^*$ denotes the dual inner product on $V$.

Indeed, this calculation proves the inner product case of Theorem \ref{thm.1}, 
part (5), as we now show.

{\sl Proof of \ref{thm.1}, part (5) for inner products:}
As in the flag case,  the \cite{PT} values are always equal to ${1 \over 2}$,
 independent of complement map, matching
the constant term of the functions in Proposition \ref{prop.mu1}.
For the two-dimensional case, we again use additivity and functoriality to 
reduce to the case in which $K$ is the cone $K_0$ from example
\ref{ex.1}.  We then note that the formula of Equation \ref{eq.ptip2d}
for the \cite{PT} construction matches the values of $\mu_0^Q(K_0)$ given in 
Example \ref{ex.2}.  
This completes the proof.

A comparison of the functions $\mu_0^Q(K_i)$ of Example \ref{ex.2} with the 
functions
$\mu_0^L(K_i)$ of Example \ref{ex.1} reveals that they are indeed different,
confirming the fact that interpolators that come from flags are different
from those that come from an inner product.
\end{example}

Let us repeat the previous example with a different triangle.

\begin{example}
\lbl{ex.3}
Consider the triangle $P$ with vertices $v_0=(0,0),v_1=(2,0),v_2=(0,1)$. 
$$
\psdraw{triangle2}{0.9in}
$$
As above, let $\Psi$ be the complement map coming from the complete flag 
$L=(L_0,L_1,L_2)$ in $V^*$ such that $L_1^*= \Span\{(d_1,d_2)\}$.
Denoting, as before, 
the tangent cones at the vertex $v_i$ by $K_i$, we see that $K_0=\Cone((1,0),(0,1))$, and $K_1=\Cone((-1,0),(-2,1))$ 
are nonsingular cones, whose $\mu$ values may be computed from Proposition 
\ref{prop.mu2}.  On the other hand, the cone $K_2=\Cone((0,-1), (2,-1))$ 
is singular.  We can compute $\mu^L(K_2)$ by subdividing $K_2$ 
into two cones 
$$
K_2=C_1 \cup C_2, \qquad 
C_1=\Cone((0,-1), (1,-1)), \qquad
C_2=\Cone((1,-1),(1,-2))
$$ 
using the ray $\rho= \Cone((1,-1))$.  The additivity of $\mu$ implies that
$$
\mu^L(K_2) = \mu^L(C_1) + \mu^L(C_2)-\mu^L(\rho).
$$
The terms on the right hand side can be computed using Propositions 
\ref{prop.mu2} and \ref{prop.mu1}.  Alternatively, one can compute 
$\mu^L(K_2)$ directly from the defining formula for interpolators.  
Letting 
$F_1= \Cone((0,-1))$ and $F_2= \Cone((2,-1))$ 
denote the one dimensional faces of $K_2$, we get
\begin{align*}
\mu^L(K_2) &= S(K_2) - I(K_2) -\mu^L(\Supp(K_2,F_1))I(F_1) 
-\mu^L(\Supp(K_2,F_2))I(F_2)\\
&= \frac{1+e^{\xi_1-\xi_2}}{(1-e^{-\xi_2})(1-e^{2\xi_1-\xi_2})}
+ \frac{2}{\xi_2 (2\xi_1-\xi_2)}
+ \frac1{-\xi_2} B\biggl(  \xi_1 -\frac{d_1}{d_2}\xi_2   \biggr)
+ \frac1{2\xi_1-\xi_2} B\biggl( \frac{d_1\xi_2-d_2\xi_1}{d_2-2d_1}  \biggr) 
\end{align*}
The constant terms of $\mu^L(K_i)$ are Morelli's rational functions
\begin{align*}
\mu_0^L(K_0)&= \frac{d_1^2+3d_1d_2 +d_2^2}{12d_1d_2} \\
\mu_0^L(K_1) &= \frac{11d_1^2-7d_1d_2 +d_2^2}{12d_1(2d_1-d_2)} \\
\mu_0^L(K_2) &= \frac{d_1^2-4d_1d_2 +2d_2^2}{-6d_2(2d_1-d_2)} 
\end{align*}
agreeing with the numbers shown in figure in the introduction.

Again, one can verify by direct computation that 
$\sum_{i=0}^{2} \mu_0^L(K_i)=1$ and 
$$
\sum_{F} \mu^L(\Supp(P,F)) I(F) = S(P) = 
1+ e^{\xi_1}+e^{2\xi_1} + e^{\xi_2}.
$$
\end{example}

\begin{example}
\lbl{ex.4}
We consider the same triangle $P$ from Example \ref{ex.3}, but now with a 
complement map $\Psi$ coming from an inner product $Q$ on $V^*$, given by 
the matrix 
$$
\left(\begin{matrix}
a & b \\
b & c
\end{matrix}\right)
$$
As in Example \ref{ex.3}, the function $\mu^Q$ of the singular cone $K_2$ 
may be computed either by subdividing, 
or by direct use of the defining property of interpolators.  We find that $\mu^Q(K_2)$ is given by
$$
\frac{1+e^{\xi1-\xi_2}}{(1-e^{-\xi_2})(1-e^{2\xi_1-\xi_2})}
+ \frac{2}{\xi_2 (2\xi_1-\xi_2)}
+ \frac1{-\xi_2} B\biggl(  \xi_1 -\frac{b}{a}\xi_2   \biggr)
+ \frac1{2\xi_1-\xi_2} B\biggl( 
\frac{-a\xi_1-2b\xi_1-b\xi_2-2c\xi_2}{a+4b+4c}  \biggr).
$$
The constant terms, in agreement with the construction of \cite{PT}, are
\begin{align*}
\mu_0^Q(K_0)&= \frac{3ac-ab-bc}{12ac} \\
\mu_0^Q(K_1) &= \frac{ab+5ac+4b^2+25bc+22c^2}{12(ac+4bc+4c^2)} \\
\mu_0^Q(K_2) &= \frac{2a^2+8ab+7ac+2b^2+2bc}{6(a^2+4ab+4bc)} 
\end{align*}
which add up to $1$. For the standard inner product, these coefficients are
$$
\mu_0^Q(K_0)= \frac{1}{4} \qquad \mu_0^Q(K_1) = \frac{9}{20}
\qquad  \mu_0^Q(K_2)=\frac{3}{10}.
$$

Again for comparison, let us detail how these numbers arise out of the 
construction of Pommersheim-Thomas.
To apply \cite[Cor.1(iv)]{PT}, we first consider the outer normal 
fan of $P$.  The rays of this fan
are $\rho_0=(1,2)$, $\rho_1=(-1,0)$, and $\rho_2=(0,-1)$.   A nonsingular 
subdivision is achieved by adding the ray
$\rho_3=(0,1)$. As before, we multiply the Todd polynomial 
$T=\sum_{i<j} D_i D_j + 
\frac{1}{12} \sum_i D_i^2$ in the ring given in \cite[Proposition 2]{PT}.  
One calculates $D_0^2= \frac{1}{5}(2D_0D_2-2D_0D_3)$, 
$D_1^2 = 0$, $D_2^2= 2D_0D_2$, and $D_3^2= -2D_0D_3$.   Hence
$$
T= \frac{1}{4} D_1D_2 + \frac{9}{20} D_0D_2 + \frac{1}{20} D_0D_3 
+\frac{1}{4} D_1D_3,
$$
Using additivity under subdivision \cite[Cor.1(iii)]{PT}, one 
finds $\mu_0(K_0)=\frac{1}{4}$, $\mu_0(K_1)=\frac{9}{20}$, and 
$\mu_0(K_2)=\frac{1}{20} + \frac{1}{4} = \frac{3}{10}$.

\end{example}

\section{Comparison with Morelli's work}
\lbl{sec.morelli}

In this section we compute values of $\nu$ and discuss the relations Morelli 
noticed
 between the constant terms of $\mu$ and $\nu$.

\subsection{Computation of $\nu^{\Psi}$ for low dimensional cones}
\lbl{sub.computenu}

For cones of dimension $0$, $1$ and $2$ it 
is easy to compute $\nu^{\Psi}(K)$ explicitly, just as we did for 
$\mu^{\Psi}(K)$ in Section \ref{sec.examples}.  The proofs in this case 
are similar.

\begin{proposition}
\lbl{prop.nu0}
Suppose that $K=\{0\}$ in the vector space  $V=\{0\}$.  Then  $\nu(K)$ is 
the constant function 1, independent of $\Psi$.  
\end{proposition}

\begin{proposition} 
\lbl{prop.nu1}
Now suppose that $V$ is one-dimensional and  $K= \Cone(v)$ is a ray in $V$ 
generated by a primitive vector $v\in\Lambda$. Then, independent of $\Psi$, 
we have
$$
\nu(K)(\xi)=B(-\langle \xi, v \rangle),
$$
\end{proposition}

\begin{proposition}
\lbl{prop.nu2line}
With the notation of Proposition \ref{prop.mu2line}, we have
$$
\nu^L(K)(\xi) = B\biggl(  -
\frac{\langle \xi, c \rangle}{\langle \rho, c \rangle}   \biggr).
$$
and
$$
\nu^Q(K)(\xi) = B\biggl( - \frac{Q( \xi, \rho )}{Q( \rho, \rho )}   
\biggr).
$$
\end{proposition}

\begin{proposition}
\lbl{prop.nu2}
With the notation of Proposition \ref{prop.mu2}, we have
$$
\nu^L(K)= 
\frac{1}{\langle \xi, v_1 \rangle \langle \xi, v_2 \rangle}
-\biggl[
 \frac{e^{\langle \xi, v_1 \rangle}}{1-e^{\langle \xi, v_1 \rangle}}
B\biggl( - 
\frac{\langle \xi, c \rangle}{\langle \rho_1, c \rangle}   \biggr)
+ \frac{e^{\langle \xi, v_2 \rangle}}{1-e^{\langle \xi, v_2 \rangle}}
B\biggl( - 
\frac{\langle \xi, c \rangle}{\langle \rho_2, c \rangle}   \biggr)+
\frac{e^{\langle \xi, v_1 \rangle}e^{\langle \xi, v_2\rangle}}{
(1-e^{\langle \xi, v_1 \rangle})
(1-e^{\langle \xi, v_2\rangle})}
\biggr]
$$
and
$$
\nu^Q(K)=
\frac{1}{\langle \xi, v_1 \rangle \langle \xi, v_2 \rangle}
-\biggl[
 \frac{e^{\langle \xi, v_1 \rangle}}{1-e^{\langle \xi, v_1 \rangle}}
B\biggl( - 
\frac{Q(\xi, \rho_1)}{Q(\rho_1, \rho_1)}   \biggr)
+ \frac{e^{\langle \xi, v_2 \rangle}}{1-e^{\langle \xi, v_2 \rangle}}
B\biggl( - 
\frac {Q(\xi, \rho_2)}{Q(\rho_2, \rho_2)}      \biggr)+
\frac{e^{\langle \xi, v_1 \rangle}e^{\langle \xi, v_2\rangle}}{
(1-e^{\langle \xi, v_1 \rangle})
(1-e^{\langle \xi, v_2\rangle})}
\biggr]
.
$$
\end{proposition}

\subsection{Morelli's coincidence in dimensions 1 and 2}
\lbl{sub.mu=nu}

In this section we will prove Morelli's observation \eqref{eq.mu=nu} for
lattice cones of dimension at most 2, using the computations of 
$\nu^{\Psi}(K)$ for cones $K$ of dimension at most $2$ from Section 
\ref{sub.computenu}.
Using these propositions, we can prove the following relation 
between $\mu$ and $\nu$ for cones of dimension at most $2$.  
The following 
theorem, proved by Morelli in the flag case, shows that the coincidence of 
constant terms holds also in the case of inner products.

\begin{theorem}
\lbl{thm.morelli2d}
Let $K$ be a cone with vertex at 0 in $V$, and assume that the dimension of 
$V$ is at most 2.  Let $\Psi$ be a complement map on $V$ induced by an 
inner product on $V^*$ or a complete flag in $V^*$ and let $\Psi^*$ be the 
corresponding complement map on $V$. Then  we have
$$
\nu^{\Psi}(K)(0) = \mu^{\Psi^*}(K^{\vee})(0).
$$
\end{theorem}

\begin{proof}
For cones of dimension $0$, the constant term of both $\mu$ and $\nu$ is 
always equal to $1$.  For cones of dimension $1$, the constant term of both 
$\mu$ and $\nu$ is always $1/2$.
Now let $K$ have dimension $2$.  
From the inversion formula of Theorem \ref{thm.6}, it follows that 
$$
\mu(K)(0) + \nu(K)(0) = {1\over 2}.
$$
Hence, the desired conclusion is equivalent to 
$$
\mu(K)(0) + \mu(K^{\vee})(0) = {1 \over 2}.
$$
Suppose that $\mu$ is induced by an inner product.  The inner product gives 
us an identification of $V$ with $V^*$, allowing us to consider $K$ and 
$K^{\vee}$ as both living in the same space $V$.   If $K=\Cone(v_1,v_2)$, 
then  $K^{\vee}=\Cone(w_1,w_2)$ with $v_i$ orthogonal to $w_i$, $i=1,2$, 
under the inner product.   Furthermore, considering the orthogonal cones 
$L_1=\Cone(v_1,w_1), L_2=\Cone(v_2, w_2)$, we have the following equation 
of characteristic functions:
$$\chi_{K_1} + \chi_{K^{\vee}} =
\chi_{L_1} + \chi_{L_2},
$$
and hence
$$\mu(K_1) + \mu(K^{\vee}) =
\mu(L_1) + \mu(L_2).
$$
Now the desired result follows from the fact that $\mu(L)(0) = 1/4$ for 
orthogonal 2-dimensional cones $L$.  To see this note that $L$ and its 
image $\rot_{\pi/2}(L)$ under rotation by $\pi/2$ form a half-space, and 
rotation is an isomorphism of $V$ that preserves the inner product, hence 
$\mu(L)(0)= \mu(\rot_{\pi/2}(L))(0)$.

Now suppose $\Psi$ is induced by a complete flag in $V^*$.  This flag is 
determined by a line $U$ in $V^*$, and $\Psi^*$ is induced by the dual flag, 
determined by the dual line $U^*$ in $V$.  Now choose an identification 
$i:V\rightarrow V^*$.  Under this identification $U$ pulls back to a line 
perpendicular to $U^*$.  It follows that
$$
\mu^{\Psi^*}(K^{\vee})(0) = \mu^{\Psi}(\rot_{\pi/2}i^{-1}(K^{\vee}))(0).
$$
But $\rot_{\pi/2}i^{-1}(K^{\vee})$ and $K$ fit together to form a half-space.  
Hence
$$\mu^{\Psi}(K)(0) + \mu^{\Psi}(\rot_{\pi/2}i^{-1}(K^{\vee}))(0) = {1\over 2}.
$$
The result follows.
\end{proof}

\begin{example}
\lbl{ex.nu}
Consider the triangle $P$ from Example \ref{ex.1} with vertices 
$v_0=(0,0),v_1=(1,0),v_2=(0,1)$. For $\Psi$ corresponding to the standard 
inner product, recall that the constant terms of the $\mu$ of the vertex 
cones are $1/2$, $3/8$, and $3/8$.  According to the discussion above, we 
have that the constant terms of the $\nu$'s are found by subtracting from 
$1/2$:
$$
\nu(K_0)(0)= \frac{1}{4} \qquad \nu(K_1)(0) = \frac{1}{8} 
\qquad  \nu(K_2)(0)=\frac{1}{8}.
$$
One can then check that the volume of this triangle is given by summing  
$\nu(F)(0)$ values times the number lattice  points in the relative 
interior of $F$.  Indeed, since the only lattice points are the vertices, 
one has
$$
\Vol(P) = {1 \over 4} + {1 \over 8} + {1 \over 8} = {1 \over 2}.
$$

For the triangle from Example \ref{ex.3}, with vertices 
$v_0=(0,0),v_1=(2,0),v_2=(0,1)$, one similarly computes the $\nu(K)$ 
constant terms  (again using the standard inner product):
$$
\nu(K_0)(0)= \frac{1}{4} \qquad \nu(K_1)(0) = \frac{1}{20}
\qquad  \nu(K_2)(0)=\frac{1}{5}.
$$
This time there is one lattice point in the relative interior of an edge.  
The corresponding $\nu$ constant term is $1/2$, according to Proposition 
\ref{prop.nu2line}. Thus, we check
$$
\Vol(P) = {1 \over 4} + {1 \over 20} + {1 \over 5} + {1 \over 2} = 1.
$$
\end{example}

\ifx\undefined\bysame
        \newcommand{\bysame}{\leavevmode\hbox
to3em{\hrulefill}\,}
\fi

\end{document}